\documentclass[11pt]{amsart}
\usepackage{amsmath,amsthm,amscd,amssymb}
\usepackage{geometry}
\usepackage{pb-diagram}
\usepackage{graphicx}

\usepackage{amssymb}

\newtheorem{thm}{Theorem}[section]
\newtheorem{lem}[thm]{Lemma}
\newtheorem{prop}[thm]{Proposition}
\newtheorem{cor}[thm]{Corollary}

\theoremstyle{definition}

\newtheorem{defn}[thm]{Definition}

\newtheorem{rem}[thm]{Remark}

\numberwithin{equation}{section}
\numberwithin{figure}{section}

\newcommand{\Z}{\mathbb{Z}}
\newcommand{\K}{\mathcal{K}}
\newcommand{\C}{\mathcal{C}}

\newcommand{\R}{\mathbb{R}}

\usepackage{pstricks}
\usepackage{pst-node}

\newcommand{\snK}{\| K\|_s}
\newcommand{\hnK}{\| K\|_H}
\newcommand{\esnK}{\| K\|^{ex}_s}
\newcommand{\topnK}{\| K\|^{top}_s}
\newcommand{\qh}{\text{quasi-homomorphism}}

\title[The geometry of the knot concordance space]{The geometry of the knot concordance space}

\author{Tim D. Cochran$^{\dag}$}
\address{Department of Mathematics MS-136, P.O. Box 1892, Rice University, Houston, TX 77251-1892}
\email{cochran@rice.edu}

\author{Shelly Harvey$^{\dag\dag}$}
\address{Department of Mathematics MS-136, P.O. Box 1892, Rice University, Houston, TX 77251-1892}
\email{shelly@rice.edu}

\thanks{$^{\dag}$Partially supported by the National Science Foundation  DMS-1309081}
\thanks{$^{\dag\dag}$ Partially supported by the National Science Foundation  DMS-0748458 (CAREER) and DMS-1309070}


\begin{document}

\date{\today}
\begin{abstract}  Most of the $50$-years of study of the set of knot concordance classes, $\C$,  has focused on its structure as an abelian group. Here we take a different approach, namely we study $\C$ as a metric space admitting  many natural geometric operators, especially satellite operators. We consider several knot concordance spaces, corresponding to different categories of concordance, and two different metrics. We establish the existence of quasi-n-flats for every $n$, implying that $\C$ admits no quasi-isometric embedding into a finite product of (Gromov) hyperbolic spaces.  We show that every satellite operator is a quasi-homomorphism $P:\C\to \C$. We show that winding number one satellite operators induce quasi-isometries  with respect to the metric induced by slice genus.   We prove that strong winding number one patterns induce isometric embeddings for  certain metrics.  By contrast, winding number zero satellite operators are bounded functions and hence quasi-contractions. These results contribute to the suggestion that $\C$ is a fractal space. We establish various other results about the large-scale geometry of arbitrary satellite operators.
\end{abstract}

\maketitle

\section{Introduction}\label{sec:intro}
A {\it classical knot} is an embedded oriented $S^1$ in $S^3$. We are interested in the ``$4$-dimensional''  equivalence relation on knots, called \textit{concordance}, due to Fox and Milnor ~\cite{FoMi57,FoMi66}.  Two knots, $K_0\hookrightarrow
S^3\times \{0\}$ and $K_1\hookrightarrow S^3\times \{1\}$, are {\it concordant} if
there is an annulus smoothly and properly embedded in $S^3\times [0,1]$ which restricts on its boundary to
 the given knots.    Knot concordance is an important microcosm for the general study of $4$-dimensional manifolds ~\cite{CF}.  Our over-arching goal is to discover more about the set, $\C $, of  concordance classes of knots.

Most of the $50$-year history of the study of $\C$ has focused on its structure as an abelian group under the operation of connected sum. Here we take a decidely different approach, namely we study $\C$ as a metric space admitting  many natural geometric operators called \textit{satellite operations}.  Since the simplest example of a satellite operator is connected-sum with a fixed knot, this approach can be argued to be more general than focusing on $\C$ as an abelian group.  In particular we address to what extent satellite operations are isometries, quasi-isometries, self-similarities or approximate self-similarities, and whether they are quasi-homomorphisms. Moreover, it  was  conjectured in ~\cite{CHL5} that $\C$ is a \textbf{fractal space}.  A fractal space is a metric space that admits systems of natural self-similarities, which in the weakest context (without a metric) are merely injective functions ~\cite[Def. 3.1]{BGN}. By endowing $\mathcal{C}$ with a metric we are able to address this conjecture in a more meaningful way.

We will, in fact, consider four different ``concordance'' equivalence relations on $\mathcal{K}$ (corresponding to different categories), with the sets of equivalence classes being denoted by $\C$,  $\C^{top}$, $\C^{ex}$, and $\C^{\frac{1}{n}}$. Here $\C^{top}$ is the (usual) set of \textit{topological knot concordance} classes wherein two knots are equivalent if there exists a collared proper topological embedding of an annulus into a topological manifold homeomorphic to $S^3\times [0,1]$ which restricts on its boundary to the given knots. $\C^{ex}$, short for $\C^{exotic}$, is the set of equivalence classes of knots where two are equivalent if they cobound a properly, smoothly embedded annulus in a smooth manifold \textit{homeomorphic} to $S^3\times [0,1]$; that is to say they are concordant in $S^3\times [0,1]$ equipped with a possibly exotic smooth structure (see ~\cite{Boyer1}\cite[Def. 2]{SatoY}\cite{CDR}). If the smooth $4$-dimensional Poincar\'{e} Conjecture is true then $\C^{ex}=\C$. Finally, for a fixed non-zero integer $n$,  let $\C^{\frac{1}{n}}$ denote the set of equivalence classes of knots in $S^3$ where two  are equivalent if they cobound a smoothly embedded annulus in a smooth $4$-manifold that is $\Z[\frac{1}{n}]$-homology cobordant to $S^3\times [0,1]$. For odd $n$ it seems to be unknown whether or not $\C=\C^{\frac{1}{n}}$!  For economy we will  sometimes often the notation $\C^{*}$ to denote either $*=\varnothing$ (i.e. $\C$), $*=top$, $*=ex$ or $*=\frac{1}{n}$. Each of these is an abelian group under connected sum, with identity the class of the trivial knot $U$, and where the inverse of $K$, denoted $-K$, is the reverse of the mirror image of $K$, denoted $r\overline{K}$.  If $K=0=U$ in $\C$ (respectively: $\C^{ex}, \C^{top}, \C^{\frac{1}{n}})$ then $K$ is called a (smooth) \textbf{slice knot} (respectively: pseudo-slice, topologically slice, $\Z[\frac{1}{n}]$-slice). This is equivalent to saying that $K$ bounds a smoothly embedded disk in a manifold diffeomorphic to $B^4$ (respectively: bounds a smoothly embedded disk in an exotic $B^4$, bounds a collared, topologically embedded disk in a manifold  homeomorphic to $B^4$, bounds a smoothly embedded disk in a smooth manifold that is $\Z[\frac{1}{n}]$-homology equivalent to $B^4$).

In Section~\ref{sec:metrics} we define a norm and a metric on a group. In Section~\ref{sec:metricsonconcordance}, we define several natural metrics on $\C^*$, each induced by a norm. To motivate these, consider that are two strategies to measure how far a knot is from the trivial knot, which is to say how far it is from bounding a disk in $B^4$. One could ask what is the least genus among all surfaces that it \textit{does} bound in $B^4$. Alternatively, one could ask what is the simplest $4$-manifold in which it bounds an embedded $2$-disk. For us these two types will be referred to as ``slice genus norms'' or ``homology norms'', respectively.  Specifically, on $\C$ the most natural and well-studied norm is that given by the \textit{slice genus of $K$}, $\|K\|_s\equiv g_s(K)$, which is the minimum genus among all compact oriented surfaces smoothly embedded in $B^4$ with boundary $K$.  The \textit{homology norm}, denoted $\|-\|_H$ is essentially the minimal $\beta_2$ of a certain class of $4$-manifolds whose boundary is the zero-framed surgery on $K$ (see Section~\ref{sec:metricsonconcordance} for details). The homology norm is only known to be a pseudo-norm on $\C$. One of our first results is that these two really are quite different.

\newtheorem*{prop:metricsaredifferent}{Proposition~\ref{prop:metricsaredifferent}}
\begin{prop:metricsaredifferent} The identity map $i: (\C,d_s)\to (\C,d_H)$ is not a quasi-isometry.
\end{prop:metricsaredifferent}

\noindent On $\C^{ex}$ these notions yield two norms. One, $\|-\|_{ex}$ is give by the ``slice genus'' in a $4$-ball with a potentially exotic smooth structure. The other is the homology norm which is a true norm on $\C^{ex}$. There are also analogues of these two metrics for the other categories  $\C^{top}$ and $\C^{\frac{1}{n}})$ but we leave the precise definitions to the body of the paper.

In Section~\ref{sec:quasi-flats} we show the existence of \textit{quasi-n-flats}  for most of these metric spaces. Recall this means:

\newtheorem*{thm:quasiflats}{Theorem~\ref{thm:quasiflats}}
\begin{thm:quasiflats} For each $n\geq 1$ there are subspaces of $(\C, d_s)$ and $(\C, d_H)$ that are quasi-isometric to $\mathbb{R}^n$. The same holds for $(\C^{ex}, d_s^{ex}), (\C^{ex}, d_H^{ex}), (\C^{top}, d_s^{top}), $ and $(\C^{top}, d_H^{top})$.
\end{thm:quasiflats}

It follows that $(\C, d_s)$ cannot be isometrically embedded in any finite product of $\delta$-hyperbolic spaces.

The proposed self-similarities are classical \textit{satellite operations}. Let $P$ be a knot in a solid torus, $K$ be a knot in $S^3$ and $P(K)$  be the satellite knot of $K$ with pattern $P$.  The algebraic intersection number of $P$ with a cross-sectional disk of the solid torus is called \textit{the winding number $w$ of $P$}. It is known that any such operator induces a satellite operator $P:\mathcal{C}^*\to \mathcal{C}^*$ on the various sets of  concordance classes of knots. The importance of satellite operations extends beyond knot theory. Such operations have been generalized to operations on $3$ and $4$-manifolds where they produce very subtle variations while fixing the homology type ~\cite[Sec. 5.1]{Ha2}. In particular winding number one satellites are closely related to Mazur $4$-manifolds  and Akbulut \textit{corks} ~\cite{AkKi3}. These can be used to alter the smooth structure on $4$-manifolds  (see for example ~\cite{AkYa1}).

Satellite operators seem to almost never be homomorphisms, but the presence of a metric allows one to ask whether they are ``close to being homomorphisms''. 
\begin{defn}\label{def:quasimorphism}  Suppose $G$ is a group and $A$ is an abelian group equipped with a (group) norm $\|-\|$.  A function $f:X\to A$ is called an \textit{quasi-homomorphism} if there exists a constant $D_f\geq 0$, called the \textit{defect} of $f$, such that, for all $x,y\in X$,
$$
\|f(xy)-f(x)-f(y)\|\leq D_f.
$$
\end{defn}

\newtheorem*{thm:quasimorphism}{Theorem~\ref{thm:quasimorphism}}
\begin{thm:quasimorphism} Any satellite operator $P:\C\to \C$ is a  quasi-homorphism with respect to either the norm given by the slice genus or the pseudo-norm $d_H$.
\end{thm:quasimorphism}
Moreover,
\newtheorem*{prop:morequasi}{Proposition~\ref{prop:morequasi}}
\begin{prop:morequasi} Suppose that $q:\C\to \R$ is a quasi-homomorphism whose absolute value gives a lower bound for some positive multiple of the slice genus. Then, for any satellite operator $P$,  the composition $q\circ P:\C\to \R$ is a quasi-homomorphism.
\end{prop:morequasi}
There are many concordance invariants, such as Levine-Tristram signatures, $\tau$, $s$ and $\epsilon$  that provide such $q$.

There has been considerable interest in whether satellite operators are injective (especially in light of the fractal conjecture). For example, it is a famous open problem as to whether or not the Whitehead double operator is \textbf{weakly injective} (an operator is called weakly injective if $P(K)=P(0)$ implies $K=0$ where $0$ is the class of the trivial knot) ~\cite[Problem 1.38]{Kirbyproblemlist}(see ~\cite{HedKirk2} for a survey).   In ~\cite{CHL5} large classes of winding number zero operators, called ``robust doubling operators'' were introduced and evidence was presented for their injectivity.  It was recently shown by Cochran-Davis-Ray that certain winding number $\pm 1$ satellite operators, called \textit{strong} winding number $\pm 1$ operators (see Section~\ref{sec:satelliteoper}), induce injective satellite operators on $\C^{ex}, \C^{top} $, and on $\C$, modulo the smooth $4$-dimensional Poincar\'{e} conjecture; while an arbitrary (non-zero) winding number $n$ satellite operator is injective on $\C^{\frac{1}{n}}$ ~\cite[Thm. 5.1]{CDR}.  We generalize their injectivity result to the following rather striking result with regard to the homology norm:

\newtheorem*{thm:mainisom}{Theorem~\ref{thm:mainisom}}
\begin{thm:mainisom}If $P$ is a strong winding number $ \pm 1$  pattern then 
$$
P:(\C,d_H)\to (\C,d_H)
$$
 preserves the pseudo-norm $d_H$ and is quasi-surjective (defined just below); so that if the $4D$-Poincar\'{e} conjecture is true then $P$ is a quasi-surjective isometric embedding of $\C$.
Moreover 
$$
P:(\C^*,d_H^*)\to (\C^{*},d_H^*)
$$
is an isometric embedding for $*=ex, top$ or $\frac{1}{n}$ and is quasi-surjective.  
\end{thm:mainisom}

This result, taken together with the recent result of A. Levine that some of these operators are far from surjective ~\cite{ALev2014}, means that these operators and their iterates map the concordance space isometrically onto proper subspaces. This establishes that $(\C^{ex},d_H)$ is a fractal metric space, as is $(\C,d_H)$ if the $4D$-Poincar\'{e} conjecture is true.

What can be said using the very important slice genus norm?  Recall:

\begin{defn}\label{def:quaisisom} If $(X,d)$ and $(Y,d')$ are metric spaces then a function $f:X\to Y$ is a \textbf{quasi-isometry} if there are constants $A\geq 1$, $B\geq 0$ and $C\geq 0$ such that
\begin{equation}\label{eq:defquasiisomfirst}
\frac{1}{A}d(x,y)-B\leq d'(f(x),f(y))\leq A~d(x,y)+B;
\end{equation}
and, for every $z\in Y$ there exists some $x\in X$ such that
\begin{equation}\label{eq:defquasiisomsecond}
d'(z,f(x)\leq C.
\end{equation}
A function that satisfies the second condition is called \textbf{quasi-surjective}. A function that satisfies only the first condition is called \textbf{a quasi-isometric embedding of $X$ into $Y$}. Note the definitions make sense even if one has only pseudo-metrics.
\end{defn}

\newtheorem*{thm:quasi}{Theorem~\ref{thm:quasi}}
\begin{thm:quasi} If $P$ is a winding number $\pm 1$ pattern  then $P:\C\to \C$ is a quasi-isometry with respect to the metric $d_s$. In fact, $P$ is a quasi-isometry for each of the metrics we discuss on each $\C^*$.
\end{thm:quasi}

By contrast, we show that any winding number zero operator is an approximate contraction. This follows from the much stronger:

\newtheorem*{prop:windzerobounded}{Proposition~\ref{prop:windzerobounded}}
\begin{prop:windzerobounded} Any winding number zero satellite operator on $(\C^*,d_*)$ is a bounded function, where $(\C^*,d_*)$ is any of the metric spaces we define.
\end{prop:windzerobounded}

Certain of our results follow from the basic result that any winding number $n$ satellite operator is within a bounded distance of the $(n,1)$-cable operator, which leads to:

\newtheorem*{prop:satellitemnotn}{Proposition~\ref{prop:satellitemnotn}}
\begin{prop:satellitemnotn} Suppose $m\neq \pm n$.   Then  no winding number $m$ satellite operator is within a bounded distance of any winding number $n$ satellite operator.
\end{prop:satellitemnotn}

\section{Norms and metrics on groups}\label{sec:metrics}

\begin{defn}\label{def:metric}  A \textbf{metric} on a set $X$ is a function $d: X\times X\to \R$ such that, for all $x,y,z\in X$
\begin{itemize}
\item [(M1)] $d(x,y)\geq 0$, and $d(x,y)=0$ if and only if $x=y$ (Positivity),
\item [(M2)] $d(x,z)\leq d(x,y)+d(y,z)$ (triangle inequality),
\item [(M3)] $d(x,y)=d(y,x)$ (Symmetry).
\end{itemize}
A function that only satisfies the first part of $(M1)$ is called a \textit{pseudo-metric}.
\end{defn}

If the set $X$ has a group structure then often metrics  induced from \textit{group norms}.

\begin{defn}\label{def:groupnorm}  A \textbf{norm} on a group $G$ is a function $\|-\|:G\to \R$ such that, for all $x,y\in G$
\begin{itemize}
\item [(GN1)] $\|x\|\geq 0$, and $\|x\|=0$ if and only if $x=e$ (Positivity),
\item [(GN2)] $\|xy\|\leq \|x\|+\|y\|$ (sub-additivity),
\item [(GN3)] $\|x^{-1}\|=\|x\|$ (Symmetry).
\end{itemize}
A function that satisfies only the first part of $(GN1)$ is called a \textit{pseudo-norm} or a \textit{semi-norm} .
\end{defn}

Beware that, even for abelian groups $G$, we do \textit{not} have the usual property (for normed vector spaces): $\|nx\|=n\|x\|$ for $n\in \Z$. This is assumed only when $n$ is a unit, as in ($GN3$).

It follows easily that:

\begin{prop} If $\|-\|$ is a norm (respectively, pseudo-norm) on the group $G$ then $d(x,y)=\|xy^{-1}\|$ is a metric (respectively, pseudo-metric) on the underlying set $G$, called the \textbf{metric induced by the norm $\|-\|$}. This metric is \textbf{right-invariant}, meaning that $d(xg,yg)=d(x,y)$ for all $g\in G$.
\end{prop}
\begin{proof} If $d(x,y)=0$ then $\|xy^{-1}\|=0$ so $x=y$ by (GN1). This gives (M1). For (M2) we have
$$
d(x,z)\equiv\|xz^{-1}\|=\|(xy^{-1})(yz^{-1})\|\leq \|xy^{-1}\|+\|yz^{-1}\|\equiv d(x,y)+d(y,z),
$$
using GN2.  For (M3) we have
$$
d(x,y)\equiv\|xy^{-1}\|=\|yx^{-1}\|\equiv d(y,x)
$$
using GN3. 
\end{proof}

We also have the following well-known consequence: 
\begin{prop}\label{prop:extraineq} If the metric  $d$ is induced from a group norm $ \|-\|$ on $G$ then
$$
-d(x,y)\leq \|x\|-\|y\|\leq d(x,y).
$$
\end{prop}
\begin{proof} By sub-additivity,
$$
\|x\|=\|(xy^{-1})y\|\leq \|xy^{-1}\| + \|y\|= d(x,y)+\|y\|.
$$
The other inequality follows from reversing the roles of $x$ and $y$ and using ($M3$).
\end{proof}

The reader may easily verify the following, which will not be used in the paper.

\begin{prop} A metric  on the underlying set of a group $G$ is induced from a group norm on $G$ if and only if it is  right-invariant.
\end{prop}

\section{Norms and metrics on concordance classes of knots}\label{sec:metricsonconcordance}

In this section, we will define various pseudo-norms on $\mathcal{K}$, the monoid of ambient isotopy classes of knots. These will induce norms (and metrics) on the  four different concordance groups: $\C$, $\C^{ex}$, $\C^{top}$, and $\C^{\frac{1}{n}}$ respectively. We will abuse notation and use $K$ for both the knot and its concordance class. We then discuss what is known about how these pseudo-metrics compare. In particular we establish that the slice genus norm and the homology norm are quite different.

\begin{subsection}{The slice genus norm and the homology pseudo-norm on $\C$}

\begin{defn}\label{def:slicegenusnorm}  The \textbf{slice genus} of $K$, denoted by $\snK$ or $g_s(K)$,  is the minimum
$g$ such that $K$ is the boundary of a smoothly embedded compact oriented surface of genus $g$ in $B^4$. 
\end{defn}

The following  is well-known.

\begin{prop} $\|-\|_s$ is a norm on the group $\C$ and so  $d_s(K,J)=\|K-J\|_s$ is a metric on the set $\C$.
\end{prop}

It is also easy to see that $d_s$ has an alternative definition:
\begin{prop} \label{prop:altdefslicedist} $d_s(K,J)$ is equal to the mimimal genus of a compact oriented surface   properly embedded in $S^3\times  [0,1]$, whose boundary is the disjoint union of $K\hookrightarrow S^3\times \{0\}$ and $-J\hookrightarrow -S^3\times \{1\}$.  
\end{prop}
\begin{proof} Suppose $\Sigma$ is such a surface  in $S^3\times  [0,1]$. Choose an arc on $\Sigma$ going from $K$ to $J$. Deleting a neighborhood of this arc leaves a $4$-ball containing an embedded surface $\Sigma'$ of the same genus, whose boundary is $K\# -J$. Hence $d_s(K,J)\leq$ genus($\Sigma$). This process is reversible so the other inequality follows.
\end{proof}

Next we define the so-called \textit{homology norm} on $\C$, which is only known to be pseudo-norm.

\begin{defn}\label{def:hnormonC}  $\hnK$, the \textbf{homology norm of} $K$, has two equivalent definitions:
\begin{enumerate}
\item $\hnK$ is the minimum of $\frac{1}{2}(\beta_2(V)+|\sigma(V)|)$ where $V$ ranges over all smooth, oriented, compact, simply-connected $4$-manifolds with $\partial V=S^3$ in which $K$ is slice, that is, in which $K$ bounds a smoothly embedded disk that represents $0$ in $H_2(V, \partial V)$  (compare ~\cite{Taylor1979});
\item $\hnK$ is the minimum of $\frac{1}{2}(\beta_2(W)+|\sigma(W)|)$ where $W$ ranges over over all smooth, compact, oriented $4$-manifolds $W$ whose boundary is $M_K$ (zero-framed surgery of $S^3$ along $K$), whose $\pi_1$ is normally generated by the meridian of $K$ and for which $H_1(M_K)\to H_1(W)$ is an isomorphism. 
\end{enumerate}
\end{defn}

We sketch the proof of the equivalence. Suppose that $K=\partial \Delta $ is slice in a $4$-manifold $V$ satisfying the first part of Definition~\ref{def:hnormonC} with $\frac{1}{2}(\beta_2(V)+|\sigma(V)|)=n$. Let $W$ be the exterior in $V$ of a tubular neighborhood of $\Delta$. Then $\partial W=M_K$ and one can easily check that $W$ has the properties of part two of Definition~\ref{def:hnormonC} and $\frac{1}{2}(\beta_2(W)+|\sigma(W)|)=n$. For the other direction, one simply adds a $2$-handle to $W$ along the meridian of $K$ to arrive at a suitable $V$.

\textbf{Remark}: It might seem more natural to the reader to minimize the quantity $\beta_2(V)$ and indeed this gives another pseudo-metric $d_{H'}$. But the inequalities:
$$
\frac{1}{2}\beta_2(V)\leq \frac{1}{2}(\beta_2(V)+|\sigma(V)|)\leq \beta_2(V)
$$
show that these metrics are quite closely related. For example it is easy to show that the identity map is a quasi-isometry $(\C,d_H)\to (\C,d_{H'})$, and that Theorem~\ref{thm:mainisom} holds for $d_{H'}$.  Our choice has the property that it gives a better approximation to  $d_s$ (see Proposition~\ref{prop:metricsarerelated} below).

\begin{prop}\label{prop:hnormisnorm} $\|-\|_H$ is a pseudo-norm on the group $\C$ and so  $d_H(K,J)=\|K-J\|_H$ is a pseudo-metric on  $\C$.  If the $4$-dimensional smooth Poincar\'{e} conjecture is true then $\|-\|_H$ is a norm on $\C$ and   $d_H(K,J)$ is a metric on  $\C$. 
\end{prop}
\begin{proof}  First we establish ($GN1$). Clearly $\hnK\geq 0$ with equality if and only if $K$  is slice in a contractible manifold, which is, by the work of Freedman, a possibly exotic $4$-ball. It follows that  if the $4$-dimensional Poincar\'{e} is true then $K=0$ in $\C$. 

Next, suppose $K$ and $J$ are slice in $4$-manifolds $V_K$ and $V_J$ satisfying the first part of Definition~\ref{def:hnormonC} with $\frac{1}{2}(\beta_2(V_k)+|\sigma(V_K)|)=\|K\|_H$ and $\frac{1}{2}(\beta_2(V_J)+|\sigma(V_J)|)=\|J\|_H$.  Then $K\# J$ is slice in the boundary connected sum of $V_K$ and $V_J$.  Hence $\|K\#J\|_H\leq \|K\|_H+\|J\|_H$,  establishing ($GN2$).

The homology norm of $K$ is clearly the same as that of  its mirror image (by reversing the orientation of $W$) and that of its reverse. Thus ($GN3$) holds.
\end{proof}

As for $d_s$, $d_H$ has an alternative definition, whose proof we leave to the reader:
\begin{prop}\label{prop:altdefhomologydistonC}  $d_{H}(K,J)$ is equal to the mimimum of  $\frac{1}{2}(\beta_2(Z)+|\sigma(Z)|)$ where $Z$ ranges over all smooth, oriented, compact, simply-connected $4$-manifolds with $\partial Z=S^3\prod -S^3$ in which $K$ and $J$ are concordant, that is in which $K\hookrightarrow S^3\times {0}$ and $-J\hookrightarrow -S^3\times{1}$ cobound  a null-homologous smoothly embedded annulus.
\end{prop}

We will need the following:

\begin{prop}\label{prop:hnormequal} If the zero framed surgeries $M_J$ and $M_K$ are homology cobordant via a cobordism whose fundamental group is normally generated by either meridian then $\hnK=\|J\|_H$.
\end{prop}

\begin{proof} Suppose $C$ is the homology cobordism and $W_K$ is a manifold with boundary $M_K$ which realizes the minimum $\frac{1}{2}(\beta_2(W_K)+|\sigma(W_K)|)$ as in the second version of Definition~\ref{def:hnormonC}. Then let $W_J=W_K\cup C$. Then $\partial(W_J)=M_J$ and $H_*(W_K)\cong H_*(W_J)$, so $\|J\|_H\leq \|K\|_H$. The result follows by symmetry.
\end{proof}

\end{subsection}

\begin{subsection}{The exotic slice genus norm and the homology norm on $\C^{ex}$}

We will discuss two norms on $\C^{ex}$, defined as above.

\begin{defn}\label{def:exslicegenusnorm}  The \textbf{exotic slice genus} of $K$, denoted by $\esnK$ or by  $g_{ex}(K)$  and is the minimum
$n$ such that $K$ is the boundary of a smoothly embedded compact oriented surface of genus $n$ in a smooth manifold that is homeomorphic to $B^4$. If the $4$-dimensional smooth Poincar\'{e} conjecture is true then $\|-\|_s^{ex}=\|-\|_s$.
\end{defn}

\begin{prop} $\|-\|_s^{ex}$ is a norm on the group $\C^{ex}$ and so  $d_s^{ex}(K,J)=\|K-J\|_s^{ex}$ is a metric on the set $\C^{ex}$. 
\end{prop}
\begin{proof} Clearly $\esnK\geq 0$, with equality if and only if $K$ is slice in a smooth $4$-manifold that is homeomorphic to $B^4$. This is equivalent to $K=0$ in $\C^{ex}$. This establishes ($GN1$). 

Suppose $K_1$ and $K_2$ bound embedded surfaces $F_1$ and $F_2$ of genera $n_1$ and $n_2$ in $4$-manifolds $\mathcal{B}_1$, $\mathcal{B}_2$,  respectively, that are homeomorphic to $B^4$. Then $K_1\# K_2$ bounds in the boundary connected sum $\mathcal{B}_1\natural\mathcal{B}_2$, the boundary connected sum of $F_1$ and $F_2$. The latter has genus $n_1+n_2$. This establishes ($GN2$).

Since the exotic slice genus of $K$ is clearly the same as that of  its mirror image and that of its reverse, ($GN3$) holds.
\end{proof}

The metric $d_s^{ex}$ has an alternative definition whose verification is the same as that of Proposition~\ref{prop:altdefslicedist}.
\begin{prop}\label{prop:altdefexdist}  $d_s^{ex}(K,J)$ is equal to the mimimal genus of a compact oriented surface   properly embedded in a smooth $4$-manifold homeomorphic to $S^3\times  [0,1]$, whose boundary is the disjoint union of $K\hookrightarrow S^3\times \{0\}$ and $-J\hookrightarrow S^3\times \{1\}$.  
\end{prop}

\begin{prop}  The homology norm, as defined in Definition~\ref{def:hnormonC}  induces a norm on the group $\C^{ex}$ and so  $d_H(K,J)=\|K-J\|_H$ induces a metric on  $\C^{ex}$. 
\end{prop}
\begin{proof}  First we establish ($GN1$). Clearly $\hnK\geq 0$ with equality if and only if $K$  is slice in a possibly exotic $4$-ball so $K=0$ in $\C^{ex}$. Otherwise the proof is identical to that of Proposition~\ref{prop:hnormisnorm}.
\end{proof}

The metric $d_H$ has an alternative definition, whose proof we leave to the reader:
\begin{prop}\label{prop:altdefhomologydist}  $d_{H}(K,J)$ is equal to the mimimum of  $\frac{1}{2}(\beta_2(Z)+|\sigma(Z)|)$ where $Z$ ranges over all smooth, oriented, compact, simply-connected $4$-manifolds with $\partial Z=S^3\prod -S^3$ in which $K$ and $J$ are concordant, that is in which $K\hookrightarrow S^3\times {0}$ and $-J\hookrightarrow S^3\times{1}$ cobound  a null-homologous smoothly embedded annulus.
\end{prop}

\end{subsection}

\begin{subsection}{The topological slice genus norm and the homology norm on $\C^{top}$}

Just as in the smooth category, there are norms on $\C^{top}$ given by the ``slice genus'' and the ``homology norm''. The proofs that these are indeed norms are straightforward.

\begin{defn}\label{def:topslicegenusnorm}  $\topnK$, also denoted $g_{top}(K)$,  the \textbf{topological slice genus} of $K$, and is the minimum
$n$ such that $K$ is the boundary of a compact oriented surface of genus $n$ topologically and flatly embedded in a topological manifold that is homeomorphic to $B^4$. This induces the metric $d_s^{top}$.
\end{defn}

\begin{defn}\label{def:tophnorm}  $\hnK^{top}$, the \textbf{topological homology norm of} $K$, has two equivalent definitions:
\begin{enumerate}
\item $\hnK^{top}$ is the minimum of $\frac{1}{2}(\beta_2(V)+|\sigma(V)|)$ where $V$ ranges over all  oriented, compact, simply-connected topological $4$-manifolds with $\partial V=S^3$ in which $K$ is a slice, that is, in which $K$ bounds a topologically flatly embedded disk that represents $0$ in $H_2(V, \partial V)$;
\item $\hnK^{top}$ is the minimum of $\frac{1}{2}(\beta_2(W)+|\sigma(W)|)$ where $W$ ranges over over all compact, oriented topological $4$-manifolds $W$ whose boundary is $M_K$, whose $\pi_1$ is normally generated by the meridian of $K$ and for which $H_1(M_K)\to H_1(W)$ is an isomorphism. 
\end{enumerate}
This induces the metric $d_H^{top}$.
\end{defn}

By the same proof as in Proposition~\ref{prop:hnormequal} we have:
\begin{prop}\label{prop:hnormtopequal} If the zero framed surgeries $M_J$ and $M_K$ are topologically homology cobordant via a cobordism whose fundamental group is normally generated by either meridian then $\hnK^{top}=\|J\|_H^{top}$.
\end{prop}

\end{subsection}

\begin{subsection}{The slice genus norm and the homology norm on $\C^{\frac{1}{n}}$}

The reader may note that we could have defined another version of $\C^{\frac{1}{n}}$, in the topological category, but we resist doing so.
We will discuss two norms on $\C^{\frac{1}{n}}$. The proofs that these are norms are straightforward.

\begin{defn}\label{def:fracslicegenusnorm}  The $\Z[1/n]$-\textbf{slice genus} of $K$, denoted by $\|K\|_{1/n}$ or by  $g_{1/n}(K)$  is a norm given by the minimum
$n$ such that $K$ is the boundary of a smoothly embedded compact oriented surface of genus $n$ in a smooth manifold that is $\Z[\frac{1}{n}]$-homology equivalent to $B^4$. Let $d_s^{1/n}$ denote the induced metric.
\end{defn}

\begin{defn}\label{def:frachnorm}  $\hnK^{\frac{1}{n}}$, the $\Z[1/n]$-\textbf{homology norm of} $K$, has two equivalent definitions:
\begin{enumerate}
\item the minimum of $\frac{1}{2}(\beta_2(V)+|\sigma(V)|)$ where $V$ ranges over all smooth, oriented, compact,  $4$-manifolds with $H_1(V;\Z[\frac{1}{n}])=0$ and $\partial V=S^3$, in which $K$ is a slice, that is, in which $K$ bounds a smoothly embedded disk that represents $0$ in $H_2(V, \partial V)$;
\item the minimum of $\frac{1}{2}(\beta_2(W)+|\sigma(W)|)$ where $W$ ranges over over all smooth, compact, oriented $4$-manifolds $W$ whose boundary is $M_K$, whose  for which $H_1(M_K;Z[\frac{1}{n}])\to H_1(W;Z[\frac{1}{n}])$ is an isomorphism. 
\end{enumerate}
Let $d_H^{1/n}$ denote the induced metric.
\end{defn}

By the same proof as in Proposition~\ref{prop:hnormequal} we have:
\begin{prop}\label{prop:hnormoneovernequal} If the zero framed surgeries $M_J$ and $M_K$ are smoothly $\Z[1/n]$-homology cobordant then $\hnK^{1/n}=\|J\|_H^{1/n}$.
\end{prop}

\end{subsection}

\begin{subsection}{Comparison of metrics}

The homology metric is not only different from  $d_s$ and $d_s^{ex}$, but is\textit{ very} different, as quantified by the following result. Surprisingly, the elements of this proof have only very recently come to light.

\begin{prop}\label{prop:metricsaredifferent} Neither of the functions $i: (\C,d_s)\to (\C,d_H)$ and $j: (\C^{ex},d_s^{ex})\to (\C^{ex},d_H)$, induced by the identity map, is a quasi-isometry.
\end{prop}
\begin{proof} It suffices to exhibit a family of knots $\{K_i~| ~ i\in \Z_+\}$  on which both the slice genus, $g_s$, and the exotic slice genus, $g_{ex}$ are unbounded functions of $i$, but on which the homology norm is a bounded function. For then there can be no constants $A$ and $B$ satisfying
\begin{equation}\label{eq:defquasiisom2}
\frac{1}{A}g_*(K_i)-B\equiv \frac{1}{A}d_s^*(K_i,U)-B\leq d_H(K_i,U)\equiv \|K_i\|_H
\end{equation}
for $*=s$ or $*=ex$ since the right-hand side is a bounded function of $i$ whereas the left-hand side is not. 

Such a family of knots was exhibited in ~\cite[Prop. 3.4]{Ray4} (building on ~\cite{CFHH}). There it was shown that for a fixed satellite operator $P$ (in fact the mirror image of the one shown in of Figure~\ref{fig:satelliteexample}), and for $T$ the right-handed trefoil knot, the family $K_i=P^i(T)$, $i\geq 0$ has the property that the slice genus and the exotic slice genus of $K_i$ are equal to $i+1$.  Therefore both $\|-\|_s$ and $\|-\|^{ex}_s$ are unbounded on this family. On the other hand, since $P$ is a strong winding number one operator with $P(U)$ unknotted, by ~\cite[Corollary 4.4]{CDR}, for any knot $J$, the zero framed surgeries $M_J$ and $M_{P(J)}$, are smoothly homology cobordant via a cobordism whose fundamental group is normally generated by either meridian. Thus, by Proposition~\ref{prop:hnormequal}, $\|J\|_H=\|P(J)\|_H$ for any knot $J$. In particular, $\|T\|_H=\|P(T)\|_H=\|P(P(T)\|_H$, et cetera. Thus $\|K_i\|_H=\|T\|_H$ for each $i$.
\end{proof}

We now compare all of the metrics. Since each $d_*$ defined above has a well-defined meaning for any pair of (\textit{isotopy} classes of) knots, their values can be compared, even though the functions $d_*$ only give metrics on the appropriate set of concordance classes.
Below, an inequality means for every pair of knots $J$ and $K$, while a strict inequality means that, in addition, there exist knots $K$ and $J$ for which the metrics differ. 

\begin{prop}\label{prop:metricsarerelated}  \begin{equation*}
\begin{array}{l}
d_H^{top}\leq d_{H}< d_s^{ex}\leq d_{s},\\
d_H^{top}\leq d_s^{top}< d_s^{ex},\\
d_H^{1/n}< d_s^{1/n}\leq d_s^{ex},\\
d_H^{1/n}\leq d_H
\end{array}
\end{equation*}
\end{prop}
\begin{proof} The first and third inequalities in the first row are obvious. For the second inequality in the first row, it suffices to show that $\|K\|_H\leq \|K\|^{ex}_s$ for all $K$. Suppose that $\|K\|^{ex}_s=g$, so that $K$ is the boundary of a smoothly embedded compact oriented surface $\Sigma$ of genus $g$ in a smooth manifold $\mathcal{B}$ that is homeomorphic to $B^4$. Choose  disjoint simple closed curves, $\{\gamma_1,...,\gamma_g\}$, on $\Sigma$ representing half of a symplectic basis for $H_1$. There exist framings of the normal bundles of these circles whose first vector field is tangent to $\Sigma$. Performing surgery on these circles using these framings transforms $\mathcal{B}$ to $V$, in which $\Sigma$ can be ambiently surgered to a disk. Hence $K$ is smoothly slice in $V$. Since each $\gamma_i$ is null-homotopic, the collection $\{\gamma_1,...,\gamma_g\}$ bounds a disjoint collection of smoothly embedded disks in $\mathcal{B}$.  Hence surgery on these circles alters the manifold by a connected-sum with either $S^2\times S^2$ or $S^2\widetilde{\times}S^2$. In either case $V$ is simply-connected, has $\beta_2(V)=2g$ and $\sigma(V)=0$. Thus $\|K\|_H\leq \|K\|^{ex}_s$. The second inequality in the second row is obvious. For the first inequality in the second row, repeat the argument of the previous paragraph. The second  inequality in the third row is obvious. For the first inequality in the third row, repeat the argument of the previous paragraph. The inequality in the fourth row is obvious.

To see that $d_{H}$ can be strictly less than  $d_s^{ex}$, let $T$ be the right-handed trefoil and let $P$ be the mirror image of the satellite operator in Figure~\ref{fig:satelliteexample}. It was shown in  ~\cite[Corollary 4.4]{CDR} that the zero framed surgeries on $T$ and $P(T)$ are homology cobordant via a cobordism whose fundamental group is normally generated by either meridian. hence by Proposition~\ref{prop:hnormequal}, $\|P(T)\|_H=\|T\|_H=1$. Thus $d_H(P(T),U)=1$. But in ~\cite[Section 3]{CFHH} it is shown that $\tau(P(T))>\tau(T)=1$. This implies that $d_s^{ex}(P(T),U)\geq 2$. That $d_s^{top}$ can be strictly less than $d_s^{ex}$ is well-known (the Whitehead double of the trefoil knot is topologically slice but not exotic slice). To see that $d_H^{1/n}< d_s^{1/n}$,  consider $T$ and its $(n,1)$-cable as explained in ~\cite[Thm. 5.1]{CFHH}. 
\end{proof}
Additionally, it is known that for $n$ even, $d_s^{1/n}< d_s^{ex}$, since the figure eight knot is slice in a $\Z[1/2]$-homology ball but is not even a topologically slice knot. This same example shows that, for $n$ even, $d_H^{1/n}< d_H$.

The others could be equalities! In particular the question of whether or not $d_s^{1/n}=d_{s}$ for $n$ an odd prime is fascinating. It ought to be true that $d_H^{top}<d_s^{top}$, but ~\cite[Section 3]{CFHH} was unable to show that $T$ and $P(T)$ are not topologically concordant. As mentioned above, if the $4$-dimensional smooth Poincare conjecture is true then $d_s=d_s^{ex}$.

\end{subsection}

\begin{subsection}{Other metrics}

There are many other norms on $\C$ that we will not consider. For example there is the minimal number of crossing changes necessary to change a knot to a slice knot (sometimes called the \textit{slicing number} ~\cite{Li8}). There is the smallest $3$-genus among all knots in the concordance class of $K$ (called the \textit{concordance genus} of $K$) ~\cite{Li2}. The \textit{stable 4-genus} is an interesting pseudo-norm ~\cite{Li11}.

\end{subsection}

\section{Existence of quasi-flats}\label{sec:quasi-flats}

If $(X,d)$ is a metric space then a \textbf{quasi-n-flat} in $X$ is  a subspace of $X$ that is quasi-isometric to $\mathbb{R}^n$, using the Euclidean metric on $\mathbb{R}^n$. We will show that $(\C, d_s)$ has quasi-n-flats for each $n$.

\begin{thm}\label{thm:quasiflats} For each $n\geq 1$ there are subspaces of $(\C, d_s)$ that are quasi-isometric to $\mathbb{R}^n$. The same holds for $(\C, d_H)$, $(\C^{ex}, d_s^{ex}), (\C^{ex}, d_H), (\C^{top}, d_s^{top}), $ and $(\C^{top}, d_H^{top})$.
\end{thm}
\begin{proof} First consider the case of $(\C, d_s)$.  It is well-known (and easy to see) that the inclusion of the integer lattice $\Z^n\hookrightarrow \R^n$ with the taxicab (or $\ell_1$) metric, $d_t$, on $\Z^n$, is a quasi-isometry. Hence it suffices to exhibit a quasi-isometric embedding
$(\Z^n, d_t)\hookrightarrow (\C,d_s)$. It is well known that $\C$ contains free abelian subgroups of arbitrarily large rank.
Specifically, suppose that $K_1,...,K_n$  are linearly-independent concordance classes, each of  slice genus one, that are detected by homomorphisms $\sigma_j:\C\to \Z$, $1\leq j\leq n$, meaning that $\sigma_j(K_i)=2\delta_{ij}$. These homomorphisms show that the free abelian group on $\{K_i\}$  is a subgroup of $\C$. Assume also that these $\sigma_j$ give lower bounds on the slice genus in the sense that
\begin{equation}\label{eq:quasi2}
g_s(K_i)\geq \frac{1}{2} |\sigma_j(K_i)|.
\end{equation}
Such classes $K_i$ can easily be found by taking a certain family of (genus one) twist-knots and considering certain Tristram signature functions as the $\sigma_j$ (see ~\cite[Theorem 2.27]{Tri}).

We will show that the embedding $(\Z^n, d_t)\hookrightarrow (\C,d_s)$ given by the $K_i$, is a quasi-isometric embedding. Suppose $\vec{x}, \vec{y} \in \Z^n$ where $\vec{x}=(x_1,...,x_n)$ and $\vec{y}=(y_1,...,y_n)$. It now suffices to show that
$$
\frac{1}{n}d_t(\vec{x},\vec{y})\leq  d_s(\vec{x}, \vec{y}) \leq n d_t(\vec{x},\vec{y}).
$$
Using the definitions of the metrics this is equivalent to:
\begin{equation}\label{eq:quasi1}
\frac{1}{n}\sum(|x_i-y_i|)\leq g_s(\sum (x_i-y_i)K_i)\leq n\sum(|x_i-y_i|).
\end{equation}
By the subadditivity and symmetry of Definition~\ref{def:groupnorm},
$$
g_s(\sum (x_i-y_i)K_i)\leq \sum(|x_i-y_i|)g_s(K_i)= \sum(|x_i-y_i|)\leq n\sum(|x_i-y_i|),
$$
which verifies the right-hand side of  inequality  ~(\ref{eq:quasi1}).
On the other hand, by ~(\ref{eq:quasi2}),  for each $j$:
$$
g_s(\sum (x_i-y_i)K_i)\geq \frac{1}{2}|\sigma_j\left(\sum (x_i-y_i)K_i\right)|=|x_j-y_j|.
$$
Hence 
$$
g_s(\sum (x_i-y_i)K_i)\geq \frac{1}{n}\sum|x_i-y_i|,
$$
which confirms the left-hand side of  inequality  ~(\ref{eq:quasi1}).

Thus we have shown that $(\C, d_s)$ admits a quasi-n-flat. The same proof works for the other cases since inequality (\ref{eq:quasi2}) is known to hold for all these other norms ~\cite[Section 1]{Taylor1979}
\end{proof}

\section{Satellite operators and other natural operators}\label{sec:satelliteoper}

In this section we review some natural operators on $\C^*$ given by taking the reverse, taking the mirror image, the connected sum with a fixed class, satellite operators, and multiplication by an integer with respect to the group structure.

Let $ST\equiv S^1\times D^2$ where both $S^1$ and $D^2$ have their usual orientations. We will always think of $ST$ as embedded in $S^3$ in the standard unknotted fashion. Suppose $P\subset ST$ is an embedded oriented circle, called a \textbf{pattern knot}, that is geometrically essential (even after isotopy $P$ has non-trivial intersection with a meridional $2$-disk). The geometric winding number of $P$, denoted $gw(P)$ is the minimum number of these intersecion points over all patterns isotopic to $P$. The \textbf{winding number} of $P$ is the algebraic number of such intersections. We say that $P$ has \textbf{strong winding number} $\pm 1$ if the meridian of the solid torus $ST$ normally generates $\pi_1(S^3-\tilde{P})$, where $\tilde{P}$ is the knot $P\subset ST\subset S^3$ ~\cite[Def. 1.1]{CDR}. Note that if $\tilde{P}$ is unknotted then winding number one is the same as strong winding number one. Suppose $K$ is an oriented knot in $S^3$ given as the image of the embedding $f_K:S^1\to S^3$. Then there is an orientation-preserving  diffeomorphism $\tilde{f}_K:S^1\times D^2\to N(K)$, where $N(K)$ is a tubular neighborhood of $K$, such that $\tilde{f}_K=f$ on $S^1\times\{0\}$ and $\tilde{f}_K$ takes the oriented meridian $\eta$ of $ST$ to the oriented meridian of $K$, and takes a preferred  longitude of $ST$, $S^1\times \{1\}$, to a preferred oriented longitude of $K$. The (oriented) knot type of the image of $P$ under $\tilde{f}_K: ST \to N(K)\hookrightarrow S^3$ is called the (untwisted) \textit{satellite} of $K$ with \textit{pattern knot} $P$  ~\cite[p. 10]{Lick2}. In this paper this will be denoted $P(K)$. 
\begin{figure}[htbp]
\setlength{\unitlength}{1pt}
\begin{picture}(252,200)
\put(50,0){\includegraphics{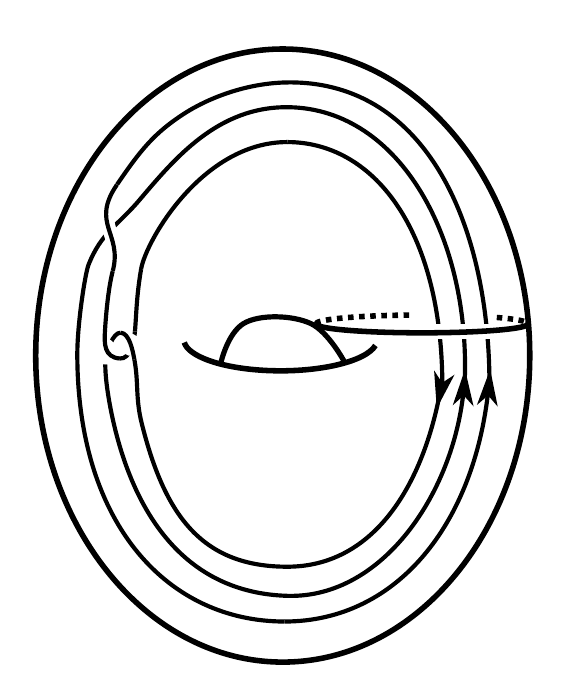}}
\put(214,108){$\eta$}
\put(120,45){$P$}
\end{picture}
\caption{A strong winding number one pattern $P$}
\label{fig:satelliteexample}
\end{figure}
Note that $\tilde{P}=P(U)$. In this paper $P$ will denote, depending on the context, either a knot in the solid torus or the corresponding induced function on a set of equivalence classes of knots in $S^3$:
$$
P:\K/\sim\to \K/\sim
$$
given by $K\mapsto P(K)$.  Such functions seem rarely to be additive with respect to the monoidal structure on $\K$ given by connected sum.  It is well known that satellite functions descend to  yield what we call \textbf{satellite operators}, on $\mathcal{K}/\sim$ for various important equivalence relations on knots, In particular any such operator descends to  $P:\mathcal{C}^*\to \mathcal{C}^*$ on the various sets of ``concordance classes of knots'' as defined in Section~\ref{sec:intro}. For a fixed pattern knot $P$, we will use the same notation for each of these satellite operators. 

Some examples are of particular importance.
\begin{defn}\label{def:consumoperator} If $J$ is a knot then \textbf{the connected-sum operator} (corresponding to $J$), denoted $C_J:\C^*\to \C^*$, is the function given by $C_J(K)=J\#K$. 
\end{defn}

Note that if $P$ is a pattern whose geometric winding number is $+1$, then  the operator $P:\C^*\to \C^*$ is $C_{\widetilde{P}}$. In particular a connected-sum operator is a satellite-operator.

\begin{defn}\label{def:revoperators} The \textbf{ reverse operator}, denoted $r:\C^*\to \C^*$, is the function that sends each knot $J$ to the class represented by, $rJ$, the reverse of the knot. 
\end{defn}

The reverse operator is a satellite operator whose pattern knot is the core of ST, oriented so that its winding number is $-1$.

\begin{defn}\label{def:ncableoperator} If $n$ is an integer  then the \textbf{(n,1)-cable operator}, denoted $C_{n,1}:\C^*\to \C^*$, is the satellite operator given by the pattern which is the $(n,1)$-torus knot. When $n=0$ it is understood that $C_{n,1}:\C^*\to \C^*$ is the \textbf{zero operator} $Z:\C^*\to \C^*$ that sends every class to the class of the unknot.
\end{defn}

There are other natural operators on $\C^*$ that are not necessarily satellite operators, but which we will consider.

\begin{defn}\label{def:mirroroperators} The \textbf{mirror image operator} is the function $J\longmapsto \overline{J}$, where   $\overline{J}$  is the mirror image. 
\end{defn}

\begin{defn}\label{def:timesmoperators} The \textbf{times m operator} is the function $J\longmapsto mJ$, where   $mJ$  denotes the connected sum of $|m|$ copies of $J$ if $m\geq 0$ or  $|m|$ copies of $-J$ if $m<0$.
\end{defn}

\section{Metric aspects of  operators}\label{sec:satellites}

The simplest operators are bijective isometries. For example:

\begin{prop}\label{prop:consumisisom} Any connected sum  operator $P$ induces a bijective isometry $P:\C^*\to \C^*$ with respect to any metric induced from a group norm. This is a quasi-isometry with respect to $A=1$, $B=0$ and $C=0$. The same holds for the ``reverse'' operator $J\mapsto rJ$ and the mirror image operator, $J\mapsto \overline{J}$, as long as $\| rJ \| =\| J \|$ (which holds for all the norms we are discussing).
\end{prop}
\begin{proof} Clearly any connected sum operator, $P$,  has an inverse which is also a connected sum operator. Thus $P$ is surjective. Then note
$$
d(P(K),P(J))=d(K\#\widetilde{P}, J\#\widetilde{P})=\| K\#\widetilde{P}\#-\widetilde{P}\#-J\|=\|K\#-J\|=d(K,J).
$$
Thus $P$ is a (surjective) isometry and so $P$ is a quasi-isometry with respect to $A=1$, $B=0$ and $C=0$.

A similar argument works in the other cases, after noting that $\|\overline{J}\|=\|r(-J)\|=\|-J\|$, by assumption, and $\|-J\|=\|J\|$ by definition of a group norm.
\end{proof}

\begin{lem}\label{lemma:nearidentity} Suppose $d$ is a metric on a group $G$ and suppose that $f:G\to G$ is a quasi-isometry with respect to constants $A$, $B$ and $C$ In the notation of Definition~\ref{def:quaisisom}) . If $g:G\to G$ is within a bounded distance $D$ of $f$ then $g$ is also a quasi-isometry, with respect to constants $A'=A$, $B'=B+2D$ and $C'=C+D$.
\end{lem}
\begin{proof} For any $x,y\in G$, by the triangle inequality we have
$$
d(g(x),g(y))\leq d(g(x),f(x))+d(f(x),f(y))+d(f(y),g(y)).
$$
Since $f$ is a a quasi-isometry with respect to constants $A$ and $B$, and $g$ is within a distance $D$ of $f$ we have:
$$
d(g(x),g(y))\leq A~d(x,y)+(B+2D).
$$
Similarly,
$$
d(f(x),f(y))\leq d(f(x),g(x))+d(g(x),g(y))+d(g(y),f(y))\leq 2D+ d(g(x),g(y)),
$$
so
$$
d(g(x),g(y))\geq d(f(x),f(y))-2C\geq \frac{1}{A}d(x,y)-(B+2D).
$$

As regards quasi-surjectivity we have:
$$
d(g(x),x)\leq d(g(x),f(x))+d(f(x),x)\leq D+C.
$$
\end{proof}

The following key result says that any satellite operator is within a bounded distance of a simpler operator.  

\begin{prop}\label{prop:boundeddistance} Any winding number $n$ operator $P:(\C^*,d_*)\to (\C^*,d_*)$ is within a bounded distance of the $(n,1)$-cable operator, with respect to any of the metrics we have introduced.
\end{prop}
\begin{proof}

The operator $P$ corresponds to a pattern knot $P\hookrightarrow ST$. The operator $ C_{n,1}$ corresponds to a different pattern knot $P'\hookrightarrow ST$, both of which have winding number $n$. Consider $P\hookrightarrow ST\times \{0\}$ and $P'\hookrightarrow ST\times \{1\}$. Since these two oriented circles in are homologous in $ST\times [0,1]$, it is easily seen that they cobound a compact oriented surface $\Sigma$ in $ST\times [0,1]$. Let $D$ be the genus of $\Sigma$. Note that $D$ depends only on $P$. Then, for any knot $J$ whose tubular neighborhood is given  by an embedding $f_J:ST\hookrightarrow S^3$, consider the image of $\Sigma$ under the map $f_J\times id: ST\times [0,1]\hookrightarrow S^3\times [0,1]$. This surface forms a smooth cobordism of genus $D$ from $P(J)$ to $P'(J)$. By Proposition~\ref{prop:altdefslicedist}, $d_s(P(J),P'(J))\leq D$.  Thus, by Proposition~\ref{prop:metricsarerelated}, for each of the metrics we have defined, $d_*(P(J),P'(J))\leq D$ for each $J$.

With more attention, one can get an explicit bound for $D$ involving only $n$ and the geometric winding number of $P$, namely $D\leq \|\widetilde{P}\|_*+1/2(gw(P)-|n|)+(|n|-1)$.
\end{proof} 

\begin{cor}\label{cor:boundeddistance} Any winding number $1$ satellite operator $P$ is within a bounded $d_*$-distance of the identity map.  Any winding number $-1$ satellite operator $P$ is within a bounded $d_*$-distance of the reverse operator. 
\end{cor}

\begin{subsection}{Winding number one satellite operators}

By Corollary~\ref{cor:boundeddistance},  a general winding number $+ 1$ operator behaves roughly like a connected sum operator. Consequently,

\begin{thm}\label{thm:quasi} If $P$ is a winding number $\pm 1$ pattern  then $P:(\C^*,d_*)\to (\C^*,d_*)$ is a quasi-isometry. 
\end{thm}
\begin{proof}  By Corollary~\ref{cor:boundeddistance}, such a $P$ is within a bounded $d_*$-distance of either the identity or the reverse operator. By Proposition~\ref{prop:consumisisom} both of the latter are bijective isometries, hence quasi-isometries with $A=1$, $B=0$ and $C=0$. Hence by Lemma~\ref{lemma:nearidentity}  $P$ is a quasi-isometry with respect to $d_*$. 
In fact, using the last line of the proof of Proposition~\ref{prop:boundeddistance}, and applying Lemma~\ref{lemma:nearidentity} where $f$ is an isometry, we see that $P$ is a quasi-isometry with respect to the constants $A=1$, $B=(gw(P)-1)$ and $C=1/2(gw(P)-1)$.  

Using a slightly different approach, one can do slightly better and get $B=1/2(gw(P)-1)$. Namely, as we will see in Figures~\ref{fig:proof2} and ~\ref{fig:proof3}  in  the proof of Theorem~\ref{thm:mainisom}, we have that 
$$
-P(K)\#P(J)=-P(K)\#P(K\#-K\#J)=R(-K\#J)
$$
where $R$ is a pattern with the same winding number, with gw($R$)$\leq$gw($P$),  and for which $\widetilde{R}$ is a slice knot. Then by Corollary~\ref{cor:boundeddistance}, $R$ is within a distance $1/2(gw(P)-1)$ of a surjective isometry. Hence
\begin{equation*}
\begin{array}{rl}
d_*(P(K),P(J))\equiv d_*(-P(K)\#P(J),U) = & d_*(R(-K\#J),U)\\
&\leq \|-K\#J\|_* + 1/2(gw(P)-1)\\
&=d_*(K,J)+1/2(gw(P)-1).
\end{array}
\end{equation*}

\end{proof}

Strong winding number one operators have an  even better behavior with respect to the homology norm.

\begin{thm}\label{thm:mainisom}If $P$ is a strong winding number $ \pm 1$  pattern then 
$$
P:(\C,d_H)\to (\C,d_H)
$$
 preserves the pseudo-norm $d_H$ and is quasi-surjective; so that if the $4D$-Poincar\'{e} conjecture is true then $P$ is an isometric embedding of $\C$ that is quasi-surjective. Moreover for $*=ex, top$ or $\frac{1}{n}$
$$
P:(\C^*,d_H^*)\to (\C^{*},d_H^*)
$$
is an isometric embedding  and is quasi-surjective.  
\end{thm}
\begin{proof} By Theorem~\ref{thm:quasi}, $P$ is a quasi-isometry and hence is quasi-surjective.

Now we show that $P$ preserves the norm (or pseudo-norm). Since any isometry is an injective function, our proof should be viewed as a generalization of ~\cite[Theorem 5.1]{CDR} where it was shown that any such operator $P$ is injective in the cases $*=ex, top$ or $\frac{1}{n}$, and injective on $\C$ if the $4D$-Poincar\'{e} conjecture holds.  First we prove the theorem in the very special case that $\widetilde{P}=0$ in $\C^{*}$ and $J=0=P(J)$:
\begin{lem}\label{lem:ribbonisom} If $R$ is a strong winding number $\pm 1$ pattern and $\widetilde{R}=0$ in $\C^{*}$ then the satellite operator $R:\C^{*}\to \C^{*}$ preserves the homology norm $\|-\|_H^*$, that is $\|R(K)\|_H^*=\|K\|_H^*$ for each $K$.
\end{lem}
\begin{proof}  Since $R$ has strong winding number $\pm 1$, by ~\cite[Corollary 4.4]{CDR} and, in the case $*=1/n$, by ~\cite[Thm. 2.1]{CFHH}, the zero framed surgeries $M_{R(K)}$ and $M_K$ are smoothly (respectively, topologically, smoothly $\Z[1/n]$-) homology cobordant. Moreover, except in the case $*=1/n$), we may assume  the cobordism has fundamental group  normally generated by each meridian. Thus, by Proposition~\ref{prop:hnormequal}, Proposition~\ref{prop:hnormtopequal} and  Proposition~\ref{prop:hnormoneovernequal},   $\|R(K)\|_H^*=\|K\|_H^*$. 

We sketch the proof of ~\cite[Corollary 4.4]{CDR} in the case $*={ex}$ for the convenience of the reader. The other cases and the proof of ~\cite[Thm. 2.1]{CFHH} are similar. There is a standard cobordism $E$ whose boundary is the disjoint union of $-M_{R(K)}$, $M_R$ and $M_K$ (see for example ~\cite[p. 2198]{CFHH}). This is obtained by gluing $M_{\widetilde{R}}\times [0,1]$ to $M_K\times [0,1]$ by identifying the surgery solid torus in $M_K\times \{1\}$ with the solid torus: $\eta\times D^2\hookrightarrow M_{\widetilde{R}}\times \{1\}$. Since $\widetilde{R}=0$ in $\C^{*}$, it bounds a slice disk $\Delta$ in some smooth homotopy $4$-ball $\mathcal{B}$. Use the manifold $\mathcal{B}-N(\Delta)$ to cap off the $M_{\widetilde{R}}$ boundary component of $E$, yielding a cobordism $V$ between $M_K$ and $M_{R(K)}$. This is the required homology cobordism.
\end{proof}

We return to the general situation in the proof of Theorem~\ref{thm:mainisom}.  Recall that, by definition,
$$
d_H^*(P(J),P(K))=\|-P(K)\#P(J)\|_H^*.
$$
Now we mimic one of the key steps in the proof of ~\cite[Theorem 5.1]{CDR}. Since $K\#-K$ is a slice knot in any category, $J=K\#-K\#J$ in $\C^{*}$, so $P(J)=P(K\#-K\#J)$ in $\C^{*}$. This last knot is pictured on the right-hand side of Figure~\ref{fig:proof2}. Hence
\begin{equation}\label{eq:proof2}
d_H^*(P(J),P(K))=\|\left(-P(K)\#\left(P(K\#-K\#J\right) \right)\|_H^*.
\end{equation}
A picture of the connected-sum of knots on the right-hand side of Equation~(\ref{eq:proof2}) is shown in Figure~\ref{fig:proof2}. The particular form we have pictured for the $-[P(K)]$ summand is not important. This form will not be used.
\begin{figure}[htbp]
\setlength{\unitlength}{1pt}
\begin{picture}(302,180)
\put(-55,0){\includegraphics{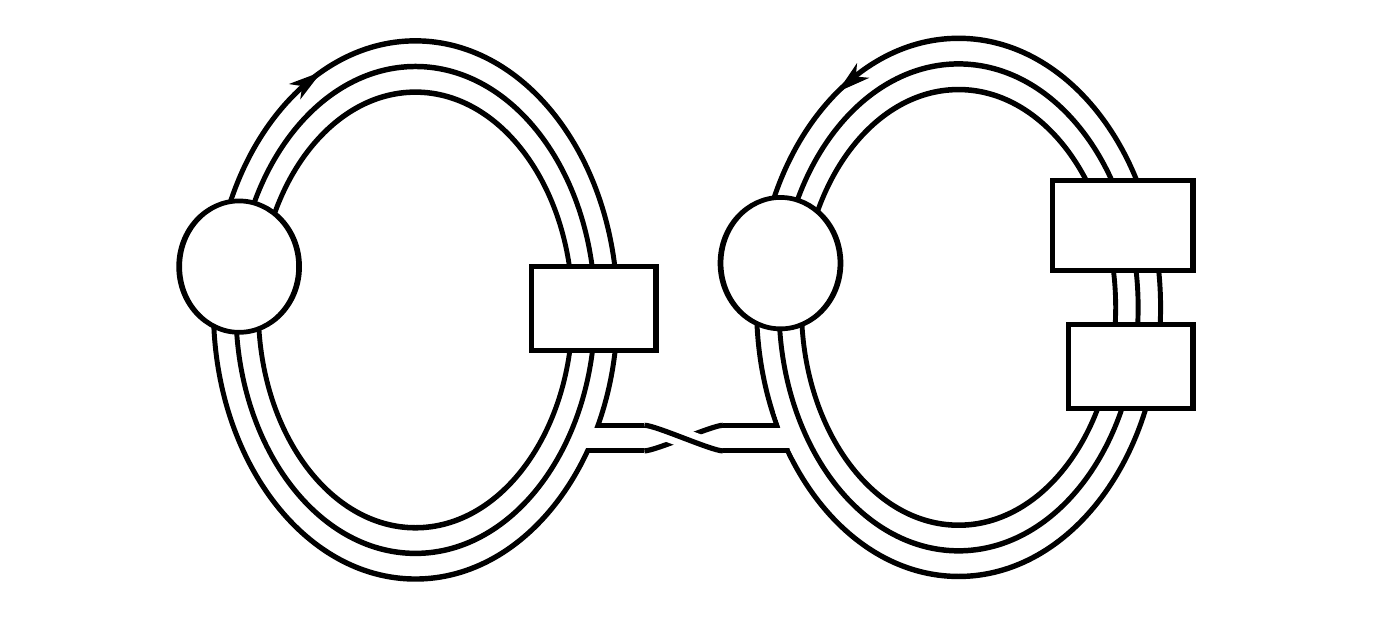}}

\put(8,104){$\overline{P}$}
\put(165,105){$P$}
\put(265,74){$K$}
\put(252,115){$-K\#J$}
\put(112,91){$\overline{K}$}
\end{picture}
\caption{}
\label{fig:proof2}
\end{figure}
Let $R$ be the pattern knot shown in Figure~\ref{fig:proof3}.
\begin{figure}[htbp]
\setlength{\unitlength}{1pt}
\begin{picture}(302,210)
\put(-27,10){\includegraphics{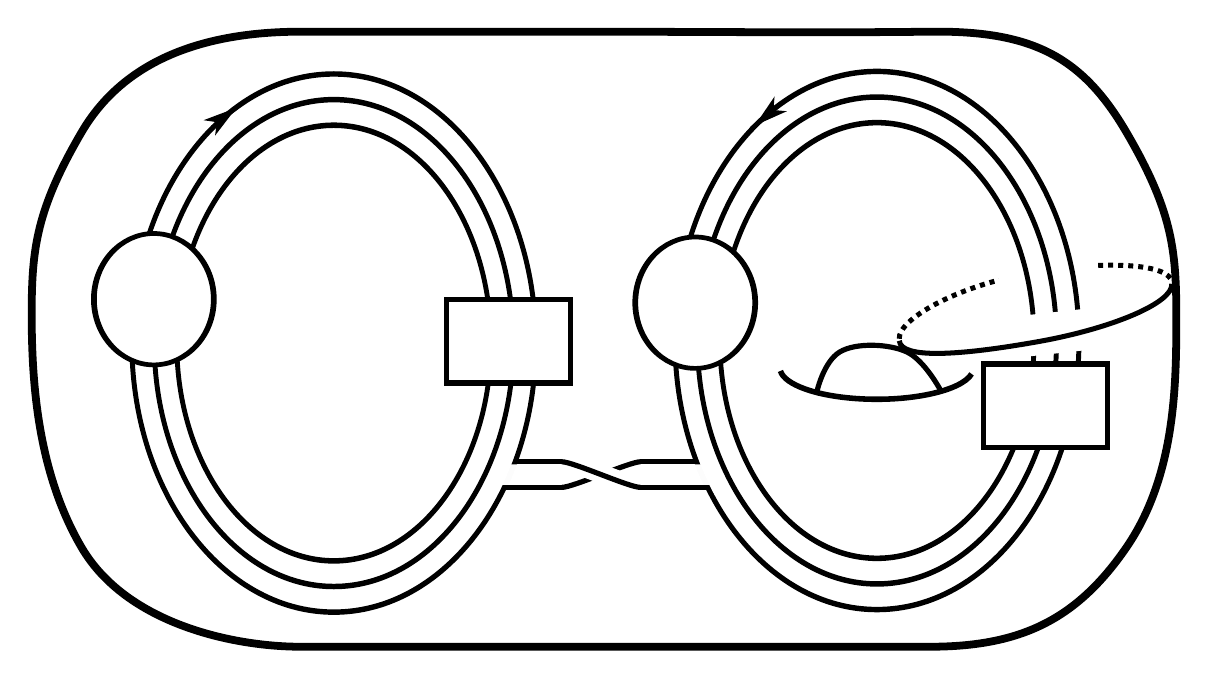}}
\put(13,114){$\overline{P}$}
\put(169,115){$P$}
\put(320,122){$\eta_R$}
\put(267,85){$K$}
\put(114,103){$\overline{K}$}
\end{picture}
\caption{The operator $R=-[P(K)]\# P(K)$}
\label{fig:proof3}
\end{figure}
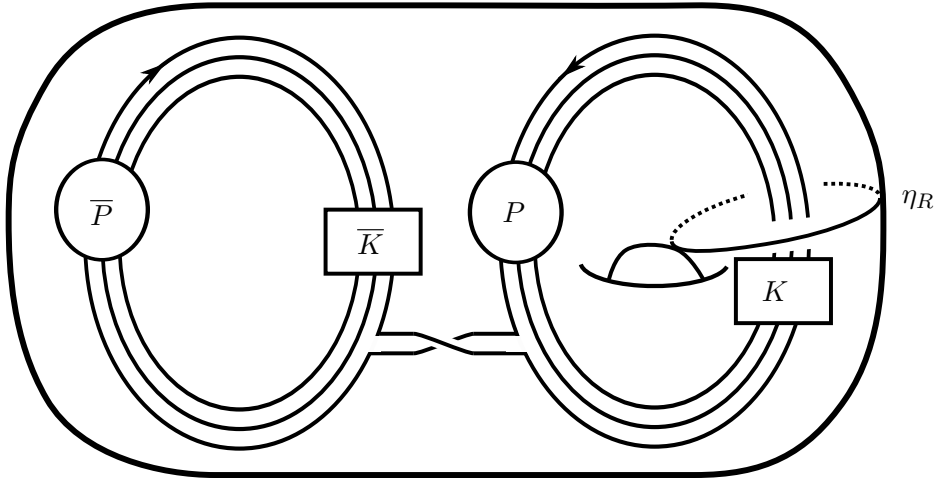
In terms of this operator Equation~(\ref{eq:proof2}) becomes
\begin{equation}\label{eq:proof3}
d_H^*(P(J),P(K))=\|R(-K\#J)\|_H^*
\end{equation}
Furthermore observe that $\widetilde{R}=-P(K)\#P(K)$ is a ribbon knot hence a slice knot in any category. Also note that the winding number of $R$ is the same as that of $P$, which is $\pm1$. In fact  it was shown in the proof of ~\cite[Theorem 5.1]{CDR}
that $R$ has strong winding number one since $P$ does. We also observe  that the geometric winding number of $R$ is at most that of $P$. Thus by Lemma~\ref{lem:ribbonisom}
$$
d_H^*(P(J),P(K))=\|R(-K\#J)\|_H^*=\|-K\#J\|_H^*\equiv d_H^*(J,K).
$$
Thus we have shown that $P$ is an isometry, hence injective. Thus $P$ is a bijection to its image. Since the induced topology is the discrete topology, all maps are continuous so $P$ is a topological embedding.

\end{proof}

\end{subsection}

\begin{subsection}{General non-zero winding number satellite operators}

\begin{thm}\label{thm:mainfracisom} If $P$ is a winding number $n$  pattern where $n\neq 0$ then  the satellite operator 
$$
P:(\C^{1/n},d_H^{1/n})\to (\C^{1/n},d_H^{1/n})
$$
is an isometric embedding. Moreover $P$ is quasi-surjective.
\end{thm}
\begin{proof} The proof is almost identical to that of Theorem~\ref{thm:mainisom} above. Note that it was already shown in ~\cite[Theorem 5.1]{CDR} that $P$ is injective. To show that $P$ preserves the norm, we repeat the proof of Theorem~\ref{thm:mainisom}, replacing Lemma~\ref{lem:ribbonisom} by the following.

\begin{lem}\label{lem:ribbonisomfrac} If $R$ is a pattern of non-zero winding number $n$ and $\widetilde{R}=0$ in $\C^{1/n}$ then the satellite operator $R:\C^{1/n}\to \C^{1/n}$ preserves the homology norm $\|-\|_H^{1/n}$.
\end{lem}
\begin{proof}  By ~\cite[Theorem 2.1]{CFHH}, the zero framed surgeries $M_{R(K)}$ and $M_K$ are smoothly $\Z[1/n]$-homology cobordant. Thus, by Proposition~\ref{prop:hnormoneovernequal},   $\|R(K)\|_H^{1/n}=\|K\|_H^{1/n}$. 
\end{proof}

\end{proof}

\end{subsection}

\begin{subsection}{Winding number zero satellite operators}

We show that every winding number zero operator is a bounded function, hence is approximately a constant map and thus is an approximate contraction. Not all winding number zero operators are injective functions. It remains possible that some are injective, or nearly so.

\begin{defn}\label{def:approxcontraction} A function $f:X\to Z$ is an approximate contraction if there is some constant $D>0$ such that, for all $x,y\in X$, $d(f(x),f(y))\leq$ max$\{D, d(x,y)\}$.
\end{defn}

Any bounded function is within a bounded distance of a constant map and hence is approximately a contraction, in that, as long as $x,y$ are not too close to each other, $d(f(x),f(y))<d(x,y)$.  Thus it follows from the $n=0$ case of Proposition~\ref{prop:boundeddistance} that:

\begin{prop}\label{prop:windzerobounded} Any winding number zero satellite operator on $(\C^*,d_*)$ is a bounded function, where $(\C^*,d_*)$ is any of the metric spaces we have defined. Thus any  winding number zero satellite operator is an approximate contraction.
\end{prop}
\begin{proof} By Proposition~\ref{prop:boundeddistance}, such a $P$ is within a bounded distance $C$ of the constant operator that sends every class to $\widetilde{P}$. Hence $d_*(P(K),P(J))\leq 2C$.
\end{proof}

We should not expect that \textit{every} winding number zero operator is an injective function, because the zero operator may be viewed as a (degenerate) satellite operator whose pattern knot is an unknot which has zero geometric winding number. But it is easy to find non-degenerate patterns that yeild the zero operator by choosing a pattern of \textit{non-zero} geometric winding number which is concordant, inside the solid torus, to this unknot. Such a pattern will be called a \textit{trivial pattern} since it induces the zero operator. There are various algebraic conditions on the pattern that ensure that an operator is non-trivial (see ~\cite[Def. 7.2]{CHL5}) and is indeed injective on very large subsets of $\C$. 

But even here the precise situation is unclear. For example, consider the family of winding number zero patterns $R^{k}$ shown on the left-hand side of Figure~\ref{fig:noninjectiveoperators}. The solid torus is the exterior of a neighborhood of the dashed circle $\eta$. This pattern has an obvious genus one Seifert surface. The $-k$ signifies the number of full twists between the two bands of that surface (without twisting the two strands of a fixed band). A knot in a box indicates that all strands passing through that box are tied into parallel copies of the indicated knot. The value of this operator on a knot $K$ is shown on the right-hand side of the figure. 
\begin{figure}[htbp]
\setlength{\unitlength}{1pt}
\begin{picture}(327,151)
\put(-20,0){\includegraphics{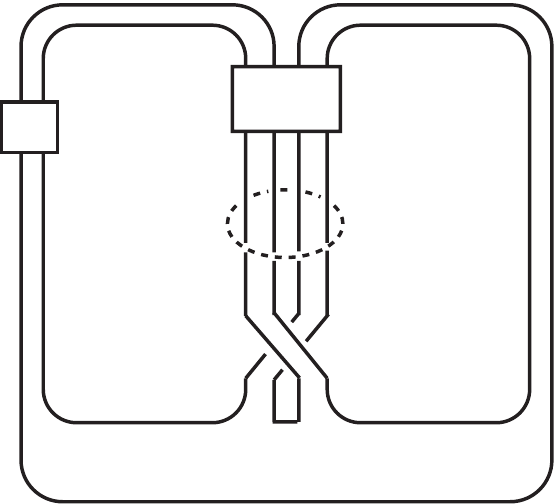}}
\put(194,0){\includegraphics{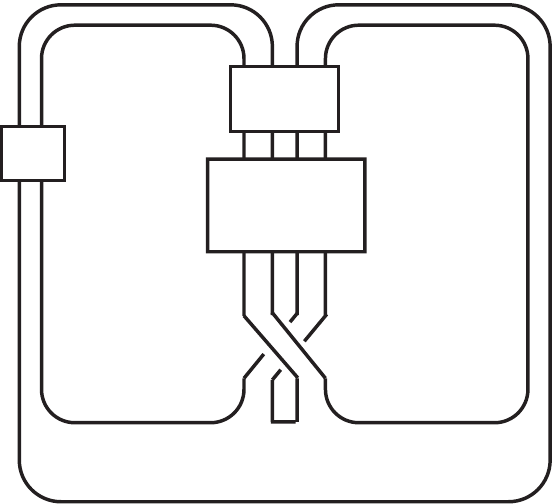}}
\put(54,113){$-k$}
\put(84,78){$\eta$}
\put(272,83){$K$}
\put(269,113){$-k$}
\put(200,98){$J$}

\put(148,56){$R^{k}(K)\equiv$}
\put(-17,106){$J$}
\put(-50,56){$R^{k}\equiv$}
\end{picture}
\caption{A family of non-trivial non-injective operators $R^{k}$}\label{fig:noninjectiveoperators}
\end{figure}
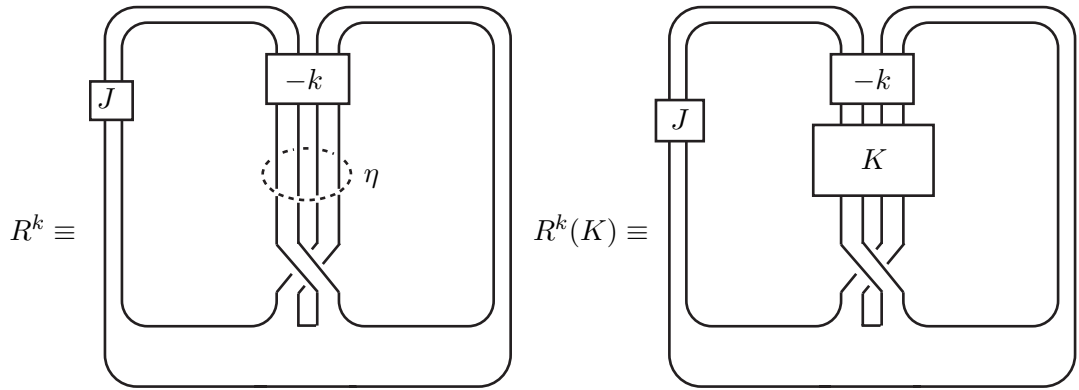
Then $R^k(U)$ is a slice knot since the core of the right-hand band has self-linking zero and has the knot type of the unknot. Similarly $R^k(\overline{J})$ is a slice knot since the left-hand band has zero self-linking and has the knot type of the ribbon knot $J\#-J$. Hence $R^k$ is not injective since both $U$ and  $\overline{J}$ are sent to $0\in \C$. Since an arbitrary pattern that is a genus one ribbon knot has two such ``metabolizing curves'', this is the generic situation. 

But recent examples indicate that the situation is even more complicated. In ~\cite{CD13} the knot, $\widetilde{R}$ in Figure~\ref{fig:DoublingOperatorBox}, was shown to be a slice knot for any knot $J$. If we let $\eta$ be a circle linking the right-hand band, the the resulting operator, $R$, has winding number zero. Then $R(U)=\widetilde{R}$ is slice. But as above there are two metabolizing curves. The core of the right-hand band of the Seifert surface for $\widetilde{R}$ has self-linking zero and has the knot type of $J_1=J_{(2,1)}\#-J$ so $R(-J_1)$ is also slice. There is a circle that goes over each band once that has self-linking zero which has a  knot type, say, $J_2$, so $R(-J_2)$ is also slice. Thus there are at least three knots (provably distinct for most $J$) that are sent to $0$ by the operator $R$.

\begin{figure}[h!]
\setlength{\unitlength}{1pt}
\begin{picture}(350,155)
\put(50,0){\includegraphics[height=2.08in]{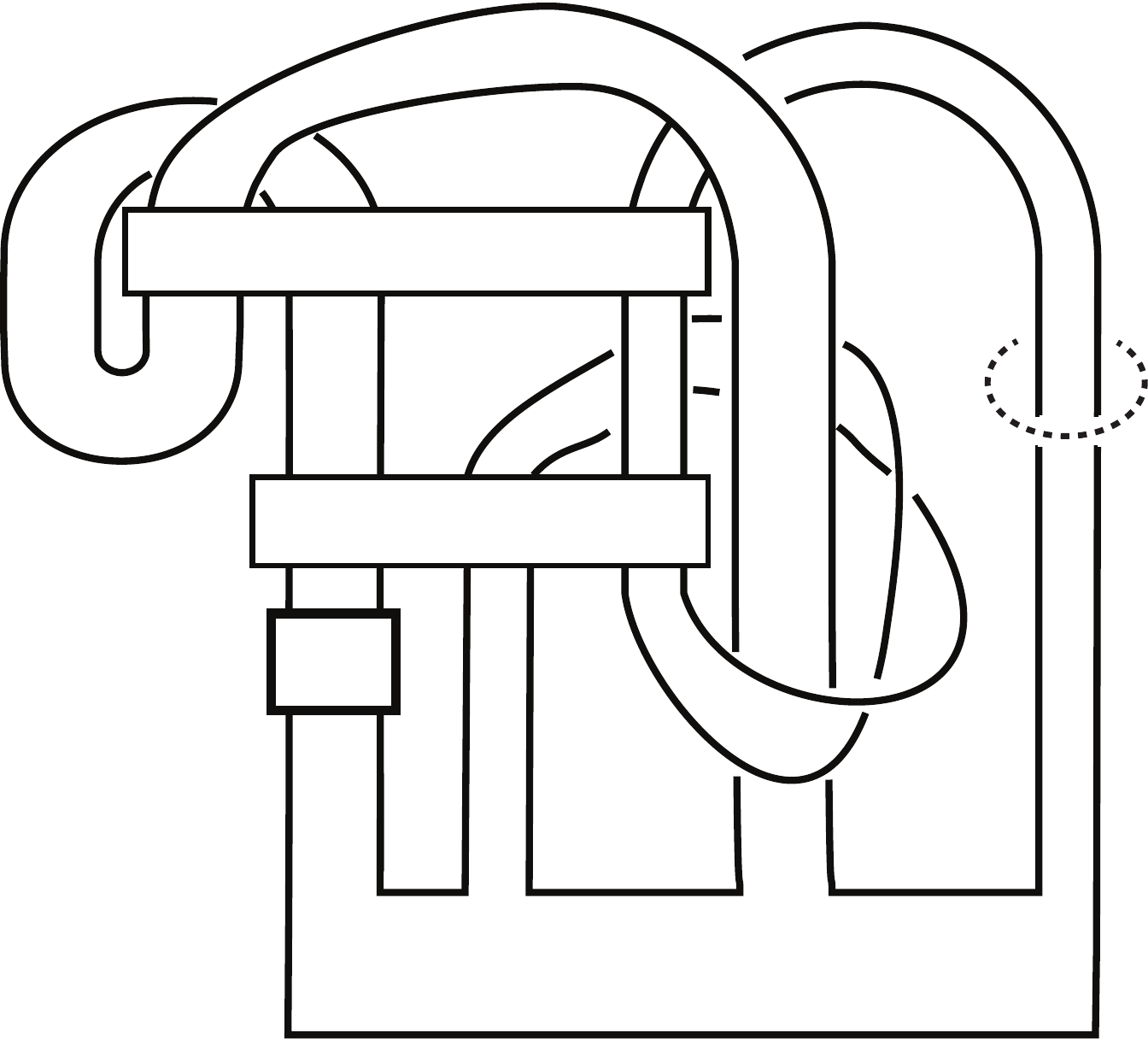}}
\put(221,95){$\eta$}
\put(100,110){$-J$}
\put(115,70){$J$}
\put(91,52){$+2$}

\end{picture}
\caption{}\label{fig:DoublingOperatorBox}
\end{figure}

Nonetheless there is a lot of evidence that some winding number zero operators are injective, or are nearly so ~\cite{CHL5}.

\end{subsection}

\begin{subsection}{Comparison of satellite operators and other natural operators}

\begin{prop}\label{prop:notlikesatellite} For $m\notin\{0, 1\}$  the function $f_m$ given by multiplication by $m$ on $\C$ (respectively on $\C^{ex}$) is not within a bounded distance of any satellite operator, where the metric is $d_s$ (respectively $d_s^{ex}$). The function $f_0$ (namely the zero operator, is not within a bounded distance of any non-zero winding number satellite operator. The identity (i.e. $f_1$) is not within a bounded distance of any satellite operator whose winding number is not $\pm 1$.
\end{prop}

\begin{proof}  The proof will be by contradiction. Suppose that $m\notin \{0,1\}$ and that $f_m$ is within $D$ of the winding number $n$ satellite operator $P$. By Proposition~\ref{prop:boundeddistance}, $P$ is within some bounded distance $D'$ of the cabling operator $C_{(n,1)}$. Let $E=D+D'$. Consequently, for every knot $J$ we would have that
$$
d(mJ,J_{(n,1)})\leq E,
$$
or, equivalently,
$$
\|-mJ\#J_{(n,1)}\|\leq E.
$$
But both the slice genus and the exotic slice genus are bounded below by one half of the absolute value of the classical knot signature. Hence to reach a contradiction we need only specify a $J$  (given $m,n$ and $\widetilde{P}$) such that
$$
|\sigma(-mJ)+\sigma(J_{(n,1)})|> 2E.
$$
We assume that $m\notin \{0,1\}$. We can assure that
$$
|\sigma(J)|>2E,
$$
by choosing $J$ to be the connected-sum of a sufficiently large  number of left-handed trefoil knots. Note that for such a $J$ the Levine-Tristram signature function takes on only two values, namely $0$ and $\sigma(J)$. Thus $\sigma(J_{(n,1)})$ is either zero (if $n=0$) or is equal to the exp($\pi/|n|$)-signature of $J$, so in any case is either zero or is equal to $\sigma(J)$. Hence we have
$$
|\sigma(-mJ)+\sigma(J_{(n,1)})|\geq |(m-1)||\sigma(J)|>2E,
$$
resulting in a contradiction.

Moreover, in the case that  $m=0$ then, as long as $n\neq 0$,  we can choose $J$ such that $|\sigma(J_{(n,1)})|$ is very large so that
$$
|\sigma(-mJ)+\sigma(J_{(n,1)})|= |\sigma(J_{(n,1)})|>2E,
$$
for a contradiction.  Hence the times zero operator is not within a bounded distance of any satellite operator with \textit{non-zero} winding number.

In the case that $m=1$ we can find a twist knot $T$ of signature $2$ wherein the roots of its Alexander polynomial have real part less than zero.  Then, if $n\neq \pm 1$, then $\sigma(T_{(n,1)})=0$. Then let $J$ be the connected sum of a large numer of such $T$, so that
$$
|\sigma(J)|> 2E.
$$
and $\sigma(J_{(n,1)})=0$. Hence the times one operator (namely the identity) is not within a bounded distance of any satellite operator whose winding number is not $\pm 1$.
\end{proof}
\begin{cor}\label{cor:notlikesatellite}  Neither the mirror image operator nor the inverse operator on $\C$ (respectively on $\C^{ex}$) is  within a bounded distance of any satellite operator, where the metric is $d_s$ (respectively $d_s^{ex}$).
\end{cor}
\begin{proof} The inverse operator is the case $m=-1$ of Proposition~\ref{prop:notlikesatellite}. Since the inverse operator is  the composition of the mirror image operator followed by the reverse operator, and since the reverse operator is both an isometry and a satellite operator, if the mirror image operator were within a bounded distance of a satellite operator then the inverse operator would be within a bounded distance of a composition of satellite operators, which is a satellite operator. This is a contradiction.
\end{proof}

The second two statements of Proposition~\ref{prop:notlikesatellite} can easily be generalized.

\begin{prop}\label{prop:satellitemnotn} Suppose $m\neq \pm n$.   Then  no winding number $m$ satellite operator is within a bounded distance of any winding number $n$ satellite operator.
\end{prop}
\begin{proof} As above, for a proof by contradiction,  it suffices to show that, given $E>0$, we can find a knot $J$ such that 
$$
\|-J_{(m,1)}\#J_{(n,1)}\|> E.
$$
For this, as above it suffices to find $\omega\in S^1$ such that 
$$
|\sigma_\omega(-J_{(m,1)})+\sigma_\omega(J_{(n,1)})|>2E,
$$
where $\sigma_\omega$ is the value of the (normalized) Levine-Tristram signature function at $\omega$. This is true because both the slice genus and the exotic slice genus are bounded below by one half of $|\sigma_\omega|$. Let $J$ be the connected sum of more than 2$E$ copies of the left-handed trefoil $-T$, and let $\alpha\in S^1$ be the root of the Alexander polynomial of the trefoil with positive imaginary part. Then, for any $\beta\in S^1$ (with positive imaginary part) whose argument is greater than that of $\alpha$, $\sigma_\beta(-T)=2$, so $\sigma_\beta(J)>2E$; whereas, for any $\beta\in S^1$ (with positive imaginary part) whose argument is less than that of $\alpha$, $\sigma_\beta(-T)=0$, so $\sigma_\beta(J)=0$.

Suppose $|n|<|m|$. Let $r_n$ and $r_m$ be the $|n|^{th}$- and $|m|^{th}$-roots of $\alpha$ that have smallest positive argument. Then $arg(r_n)>arg(r_m)$. Let $\omega\in S^1$ have argument between these and be close to $r_m$. Then 
$$
\sigma_\omega(J_{(n,1)})=\pm\sigma_{\omega^{|n|}}(J)=0
$$
since the argument of $\omega^{|n|}$ is less than that of $\alpha$. On the other hand
$$
|\sigma_\omega(J_{(m,1)})|=|\pm\sigma_{\omega^{|m|}}(J)|>2E
$$
since the argument of $\omega^{|m|}$ is slightly greater than that of $\alpha$. These choices lead to the desired contradiction.
\end{proof}

\begin{rem} With more work, the $\pm$ in Proposition~\ref{prop:satellitemnotn} should be able to be replaced by $+$.  It has been shown, using Casson-Gordon invariants that some knots are not concordant to their reverses ~\cite{KiLi1999}\cite[Thm. 5.5.2]{CollinsThesis}\cite{Naik1}. It is also known that Casson-Gordon invariants, in a certain sense, give lower bounds for the slice genus. Thus it may be able to be shown, with some effort beyond what is in the literature, that  the reverse operator is not within a bounded distance of the identity. This would imply that no winding number $-1$ satellite operator is within a bounded distance of any  winding number $+1$ satellite operator. For $n>1$, to show that no winding number n satellite operation is within a bounded distance of any winding number $-n$ operation, would seem to require calculations (of, say, Casson-Gordon invariants) far beyond what is currently in the literature. Nonetheless there is little doubt, in the authors' opinion, as to its veracity.
\end{rem}

\end{subsection}

\section{Satellite operators as quasi-morphisms}\label{sec:quasimorpism}

\begin{thm}\label{thm:quasimorphism} Any satellite operator $P:\C\to \C$ is a  quasi-homomorphism with respect to norm given by the slice genus, $\|-\|_s$. Similarly, any satellite operator $P:\C^{ex}\to \C^{ex}$ is a $\qh$ with respect to either $d_s^{ex}$ or $d_H$. 
\end{thm}
\begin{proof} The following is well-known.
\begin{lem}\label{lem:boundedquasi} If  $f:G\to (A,d)$ is a quasi-homomorphism and $g: G\to A$ is within a bounded distance of $f$ then $g$ is a quasi-homomoprhism.
\end{lem}
\begin{proof}[Proof of Lemma~\ref{lem:boundedquasi}] By sub-additivity, $\|a\|\leq \|a-b\|+\|b\|$. The lemma now follows quickly from setting $a=g(xy)-g(x)-g(y)$ and $b=f(xy)-f(x)-f(y)$.
\end{proof}
By Proposition~\ref{prop:boundeddistance}, a satellite operator of winding number $n$ is within a bounded distance of the operator $C_{n,1}$, so by Lemma~\ref{lem:boundedquasi}, it suffices to show that  latter is a $\qh$.  This is done geometrically. By adding $n$ carefully placed bands, the $(n,1)$-cable of the knot $K\# J$ can be transformed to the disjoint union of the $(n,1)$-cable of $K$ with an $n$-component link consisting of parallel copies of $J$. By adding $n-1$ more bands, the latter can be transformed to the $(n,1)$-cable of $J$. Adding one more band transforms the disjoint union to the connected sum. Hence there is a cobordism in $S^3\times [0,1]$ from the $(K\# J)_{(n,1)}$ to $K_{(n,1)}\# J_{(n,1)}$ whose genus is independent of $K$ and $J$.  Hence
$$
\|(K\# J)_{(n,1)}-K_{(n,1)}- J_{(n,1)}\|_s\equiv d_s\left((K\# J)_{(n,1)}, K_{(n,1)}\# J_{(n,1)}\right)
$$
is bounded by a function of $n$ alone. Thus the $(n,1)$-cable operator is a $\qh$.
\end{proof}

\begin{prop}\label{prop:morequasi} Suppose that $q:\C\to \R$ is a quasi-homomorphism whose absolute value gives a lower bound for a positive multiple of the slice genus, that is, there is some $C>0$ such that, for all $K$,  $|q(K)|\leq Cg_s(K)$. Then, for any satellite operator $P$,  the composition $q\circ P:\C\to \R$ is a quasi-homomorphism.
\end{prop}

\begin{proof}  The proof follows that of ~\cite[Cor. 1B]{BrKe14} where the slightly stronger assumption was made that $q$ is Lipshitz.  Let $E=P(K\#J)-P(K)-P(J)$.  Since $P$ is a quasi-homomorphism, $\|E\|_s\leq D_P$. Then
\begin{equation*}
\begin{array}{rl}
|q\left(P(K\# J)\right)-q\left(P(K)\right)-q\left(P(J)\right)|=&|q\left(P(K)+P(J)+E)\right)-q\left(P(K)\right)-q\left(P(J)\right)|\\
&\leq 2D_q+|q(E)|\\
&\leq 2D_q+ C\|E\|_s\\
&\leq 2D_q+CD_P.
\end{array}
\end{equation*}

\end{proof}
\bibliographystyle{plain}
\bibliography{mybib9}

\end{document}